\documentclass{amsart}
\usepackage[T1]{fontenc}

\usepackage{amsmath, amsthm, amssymb, epsfig, latexsym, amsfonts, mathtools, 
enumerate, extarrows}
\usepackage[alphabetic]{amsrefs}
\usepackage{tikz}
\usetikzlibrary{cd}

\usepackage[textsize=footnotesize,textwidth=20ex]{todonotes}

\usepackage{hyperref}
\usepackage{dsfont}
\AtBeginDocument{\let\mathbb\mathds}
\usepackage{lmodern}

\newtheorem{theorem}{Theorem}[section]
\newtheorem{proposition}[theorem]{Proposition}
\newtheorem{lemma}[theorem]{Lemma}

\theoremstyle{definition}
\newtheorem{definition}[theorem]{Definition}
\newtheorem{example}[theorem]{Example}
\newtheorem{notation}[theorem]{Notation}

\newtheorem{remark}[theorem]{Remark}

\mathchardef\mhyphen="2D 

\newcommand{\CAT}{\mathcal{C}}
\newcommand{\DEC}{\Phi\mhyphen\mathbf{Dec}_R}
\newcommand{\GDEC}{\Gamma\mhyphen\mathbf{Dec}_R}
\newcommand{\AXDEC}{(\Phi, \lambda)\mhyphen\mathbf{AxDec}_R}
\newcommand{\ALG}{\mathbf{Alg}_R}
\newcommand{\GRP}{\mathbf{Grp}}
\newcommand{\SET}{\mathbf{Set}}
\newcommand{\FUS}{\mathbf{Fus}}

\newcommand{\dash}{\nobreakdash-\hspace{0pt}}

\newcommand{\onto}{\mathrel{\mathrlap{\rightarrow}\rightarrow}}
\newcommand{\conj}[2]{\prescript{#1}{}{#2}}

\newcommand{\K}{\mathcal{K}}
\newcommand{\I}{\mathcal{I}}
\newcommand{\J}{\mathcal{J}}
\let\L\relax
\newcommand{\L}{\mathcal{L}}
\newcommand{\X}{\mathcal{X}}
\newcommand{\Y}{\mathcal{Y}}
\let\U\relax
\newcommand{\U}{\mathcal{U}}

\newcommand{\Z}{\mathbb{Z}}
\let\C\relax
\newcommand{\C}{\mathbb{C}}
\newcommand{\Sym}{\mathrm{Sym}}

\newcommand{\PSL}{\mathrm{PSL}}
\newcommand{\SL}{\mathrm{SL}}

\DeclareMathOperator{\ad}{ad}
\DeclareMathOperator{\Aut}{Aut}
\DeclareMathOperator{\GL}{GL}
\DeclareMathOperator{\Hom}{Hom}
\DeclareMathOperator{\Irr}{Irr}
\DeclareMathOperator{\id}{id}
\DeclareMathOperator{\Miy}{Miy}
\newcommand{\UMiy}{\widehat\Miy}
\DeclareMathOperator{\Char}{char}

\DeclareMathOperator*{\freeprod}{\raisebox{-2pt}{\scalebox{2}{$\ast$}}}

\begin{document}
\title{Decomposition algebras and axial algebras}

\author[T. De Medts]{Tom De Medts}
\address{Tom De Medts \\
    Department of Mathematics: Algebra and Geometry \\
	Ghent University \\
	9000 Gent, Belgium}
\email{tom.demedts@ugent.be}

\author[S. F. Peacock]{Simon F. Peacock}
\address{Simon F. Peacock,
    School of Mathematics \\
    Department of Mathematics \\
    University of Manchester \\
    Oxford Road \\
    Manchester \\
    M13 9PL \\
    UK \\
    and the Heilbronn Institute for Mathematical Research \\
    Manchester \\
    UK}
\email{simon.peacock@bristol.ac.uk}

\author[S. Shpectorov]{Sergey Shpectorov}
\address{Sergey Shpectorov,
    School of Mathematics \\
    Watson Building \\
    University of Birmingham \\
    Edgbaston \\
    Birmingham \\
    B15 2TT \\
    UK}
\email{s.shpectorov@bham.ac.uk}

\author[M. Van Couwenberghe]{Michiel Van Couwenberghe}
\address{Michiel Van Couwenberghe,
    Department of Mathematics: Algebra and Geometry \\
	Ghent University \\
	9000 Gent, Belgium \\
	Ph. D. fellowship of the Research Foundation -- Flanders (FWO)}
\email{michiel.vancouwenberghe@ugent.be}

\subjclass[2010]{}

\date{\today}

\begin{abstract}
	We introduce decomposition algebras as a natural generalization of axial algebras, Majorana algebras and the Griess algebra. They remedy three limitations of axial algebras:
	(1) They separate fusion laws from specific values in a field, thereby allowing repetition of eigenvalues;
	(2) They allow for decompositions that do not arise from multiplication by idempotents;
	(3) They admit a natural notion of homomorphisms, making them into a nice category.
	
	We exploit these facts to strengthen the connection between axial algebras and groups. In particular, we provide a definition of a universal Miyamoto group which makes this connection functorial under some mild assumptions.
	
	We illustrate our theory by explaining how representation theory and association schemes can help to build a decomposition algebra for a given (permutation) group. This construction leads to a large number of examples.
	
    We also take the opportunity to fix some terminology in this rapidly expanding subject.
\end{abstract}

\maketitle

\section{Introduction}\label{se:intro}

In 1982, Robert Griess proved the existence of the Monster group by constructing a $196\,884$-dimensional non-associative algebra over $\mathbb{R}$, called the \textit{Griess algebra} \cite{Gri82}. A peculiar feature of these algebras is the existence of many idempotents with the property that multiplication by each of these idempotents gives rise to a \textit{decomposition} of the algebra obeying a very precise \textit{fusion law}.

Igor Frenkel, James Lepowsky and Arne Meurmann observed that other algebras similar the Griess algebra can be retrieved as weight-2 components of certain \emph{vertex operator algebras} (VOAs) \cite{FLM88}. In an attempt to axiomatize such algebras, Alexander Ivanov introduced \textit{Majorana algebras}, a large class of real non-associative algebras obeying the same fusion law as the Griess algebra.

Only recently, in 2015, the more general concept of \textit{axial algebras} was introduced by Jonathan Hall, Sergey Shpectorov and Felix Rehren \cite{HRS15a}. Axial algebras are defined over an arbitrary field and have as defining feature that they are generated by idempotents that again give rise to decompositions satisfying a fusion law, which is now allowed to take a much more general shape. The subject has received a lot of attention since then and developed connections as far afield as the regularity theory of some classes of elliptic type PDEs and algebraic solutions of eiconal and minimal surface equations \cites{Tka2019a,Tka2019b}. See also the earlier book \cite{NTV2014}.

In May 2018, a specialized workshop on axial algebras took place at the University of Bristol funded by the Heilbronn Institute for Mathematical Research. It became apparent at this workshop that there is a need for a more general framework to study axial algebras. New observations forced us to generalize the definition even further and to separate fusion laws from the field.
At the same time, we noticed that the crucial aspect of an axial algebra is the existence of the corresponding decompositions, and not so much the fact that these arise from idempotents.

The \textit{decomposition algebras} that we introduce in this paper aim to provide a natural generalization of axial algebras that take all these facts into account.
Our hope that this is a useful framework is further emphasized by the fact that these decomposition algebras form a nice category (in contrast to the setting of axial algebras, where the natural notion of homomorphisms gives rise to a less powerful category).

\smallskip

We begin our paper by introducing (general) \textit{fusion laws} that no longer depend on a ring or field (section~\ref{se:fusionlaws}). 

In section~\ref{se:gradings}, we introduce \textit{gradings} as morphism between fusion laws and group fusion laws. This will be an essential ingredient to make the connection between (axial) decomposition algebras and groups later on. We also explain how to construct such gradings for a given fusion law.

In section~\ref{se:decalg}, we introduce \textit{decomposition algebras}. These algebras axiomatize the essence of Griess algebras, Majorana algebras and axial algebras. We believe that this definition is the right approach to study all known algebras that are reminiscent of axial algebras. Moreover, it is the first definition in this context that allows for a suitable definition of a homomorphism and hence fits into a categorical framework. We explore this framework thoroughly in Appendix~\ref{se:appendix}.

In section~\ref{se:axial}, we explain how axial algebras fit into this framework by defining \textit{axial decomposition algebras} and homomorphisms between axial decomposition algebras.

The important connection between decomposition algebras and groups is discussed in section \ref{se:Miy}, which is the longest section of the paper. We explain why the ``obvious'' connection (the \textit{Miyamoto group}) is not functorial. However, we introduce a more universal connection (the \textit{universal Miyamoto group}) which turns out to be functorial under some mild conditions. This is the subject of Proposition \ref{prop:surj-functor} and Theorem \ref{thm:axial-functor}.

In section \ref{se:rep} and section \ref{se:Nor}, we present an important source of examples of decomposition algebras for a given (permutation) group. This is very closely related to \textit{representation theory} and to the theory of \textit{association schemes} via \textit{Norton algebras}.

\medskip

\paragraph*{\textbf{Acknowledgments}}

We thank the referees for their valuable and insightful comments. We also thank Jon Hall for his suggestions on how to improve the exposition of the paper.

\smallskip

\begin{notation}\label{not:conv}
    We will use functional notation for our maps and morphisms, i.e., when $\varphi \colon A \to B$ is a map, we denote the image of an element $a$ by $\varphi(a)$.
    Consequently, we will also denote conjugation of group elements on the left:
    \[ \conj{g}{h} := g h g^{-1} . \]
\end{notation}

\clearpage

\section{Fusion laws}\label{se:fusionlaws}

In this section, we define (general) fusion laws. In contrast to previous definitions, these will no longer depend on a ring or a field.

\begin{definition}
    A \emph{fusion law}%
    \footnote{In earlier papers on axial algebras, this was referred to as ``the fusion rules'', leading to singular/plural problems. It has also been referred to as a ``fusion table''.}
    is a pair $(X, *)$ where $X$ is a set%
    \footnote{The set $X$ is often, but not always, a finite set.}
    and $*$ is a map from $X \times X$ to $2^X$, where $2^X$ denotes the power set of $X$.
    A fusion law $(X, *)$ is called \emph{symmetric} if $x * y = y * x$ for all $x,y \in X$.
\end{definition}

\begin{definition}\label{def:fus-unit}
    Let $(X, *)$ be a fusion law and let $e \in X$.
    \begin{enumerate}[(i)]
        \item We call $e$ a \emph{unit} if $e * x \subseteq \{ x \}$ and $x * e \subseteq \{ x \}$ for all $x \in X$.
        \item We call $e$ \emph{annihilating} if $e * x = \emptyset$ and $x * e = \emptyset$ for all $x \in X$.
        \item We call $e$ \emph{absorbing} if $e * x \subseteq \{ e \}$ and $x * e \subseteq \{ e \}$ for all $x \in X$.
    \end{enumerate}
\end{definition}

\begin{lemma}
    Let $(X, *)$ be a fusion law.
    If $e,f \in X$ are units with $e \neq f$, then $e * f = \emptyset$.
\end{lemma}
\begin{proof}
    We have both $e * f \subseteq \{ e \}$ and $e * f \subseteq \{ f \}$.
\end{proof}

\begin{example}[Jordan fusion law]\label{ex:jordan}
	Consider the set $X = \{ e, z, h \}$ with the symmetric fusion law
	\begin{center}
		\renewcommand{\arraystretch}{1.2}
		\setlength{\tabcolsep}{0.75em}
		\begin{tabular}[b]{c|ccc}
			$*$	& $e$ & $z$ & $h$ \\
			\hline
			$e$ & $\{ e \}$ & $\emptyset$ & $\{ h \}$ \\
			$z$ & $\emptyset$ & $\{ z \}$ & $\{ h \}$ \\
			$h$ & $\{ h \}$ & $\{ h \}$ & $\{ e,z \}$
		\end{tabular}
	\end{center}
    Here both $e$ and $z$ are units and accordingly $e*z=\emptyset$.
\end{example}

\begin{example}[Ising fusion law]\label{ex:Ising}
	Consider the set $X = \{ e, z, q, t \}$ with the symmetric fusion law
	\begin{center}
		\renewcommand{\arraystretch}{1.2}
		\setlength{\tabcolsep}{0.75em}
		\begin{tabular}[b]{c|cccc}
			$*$	& $e$ & $z$ & $q$ & $t$ \\
			\hline
			$e$ & $\{ e \}$ & $\emptyset$ & $\{ q \}$ & $\{ t \}$ \\
			$z$ & $\emptyset$ & $\{ z \}$ & $\{ q \}$ & $\{ t \}$ \\
			$q$ & $\{ q \}$ & $\{ q \}$ & $\{ e,z \}$ & $\{ t \}$ \\
			$t$ & $\{ t \}$ & $\{ t \}$ & $\{ t \}$ & $\{ e,z,q \}$
		\end{tabular} 
	\end{center}
    Again, both $e$ and $z$ are units.
\end{example}

\begin{remark}
	A fusion law $(X, *)$ can also be viewed as a map $\omega \colon X \times X \times X \to \{ 0, 1 \}$,
	where we define $\omega(x, y, z) = 1 \iff z \in x * y$.
	As such, it is clear that there is an action of $\Sym(3)$ on the set of all fusion laws.
	It turns out that the Jordan fusion law and the Ising fusion law are invariant under this action.
\end{remark}

\begin{definition}\label{def:FUS}
    Let $(X, *)$ and $(Y, *)$ be two fusion laws.
    A \emph{morphism} from $(X, *)$ to $(Y, *)$ is a map $\xi \colon X \to Y$ such that
    \[ \xi(x_1 * x_2) \subseteq \xi(x_1) * \xi(x_2) \]
    for all $x_1, x_2 \in X$,
    where we have denoted the obvious extension of $\xi$ to a map $2^X \to 2^Y$ also by $\xi$.

    This makes the set of all fusion laws into a category $\FUS$.
\end{definition}

\begin{definition}\label{def:fus-products}
    Let $(X, *)$ and $(Y, *)$ be two fusion laws.
    \begin{enumerate}[(i)]
        \item
            We define the \emph{product} of $(X, *)$ and $(Y, *)$ to be the fusion law $(X \times Y, *)$ given by
            \[ (x_1, y_1) * (x_2, y_2) := \{ (x,y) \mid x \in x_1 * x_2, y \in y_1 * y_2 \} . \]
        \item
            We define the \emph{union} of $(X, *)$ and $(Y, *)$ to be the fusion law $(X \cup Y, *)$,
            where $*$ extends the given fusion laws on $X$ and $Y$ and is defined by
            \[ x * y := \emptyset \]
            for all $x \in X$ and all $y \in Y$.
    \end{enumerate}
\end{definition}

\begin{proposition}
    The product and coproduct in the category $\FUS$ are given by the product and union of fusion laws, respectively,
    as defined in Definition~\ref{def:fus-products}.
\end{proposition}
\begin{proof}
    This follows easily from the definitions.
    Notice, in particular, that for given fusion laws $(X, *)$ and $(Y, *)$,
    the projection maps $X \times Y \to X$ and $X \times Y \to Y$ and the inclusion maps
    $X \to X \cup Y$ and $Y \to X \cup Y$ indeed induce morphisms in $\FUS$ as in Definition~\ref{def:FUS}.
\end{proof}

An important class of fusion laws are the group fusion laws.

\begin{definition}\label{def:groupfusionlaw}
	Let $\Gamma$ be a group.
	Then the map
	\[ * \colon \Gamma \times \Gamma \to 2^\Gamma \colon (g, h) \mapsto \{ g h \} \]
	is a \emph{group fusion law}.
	The identity element of $\Gamma$ is the unique unit of the fusion law $(\Gamma, *)$.
\end{definition}

\begin{remark}
	The category $\GRP$ of groups is a \emph{full} subcategory of $\FUS$:
	if $\Gamma$ and $\Delta$ are groups, then the fusion law morphisms from $(\Gamma, *)$ to $(\Delta, *)$
	are precisely those arising from homomorphisms from $\Gamma$ to $\Delta$.
\end{remark}

Two further examples of fusion laws arising in group theory and 
representation theory are given in the following examples.

\begin{example}[Class fusion law]\label{ex:classlaw}
	Let $G$ be a group with a finite number of conjugacy classes and let $X$ be the set of those conjugacy classes.
	Then we can define a fusion law on $X$ by declaring
	\[ E \in C * D \iff E \cap CD \neq \emptyset , \]
	where $CD$ is the setwise product of $C$ and $D$ inside $G$.
	The trivial conjugacy class $\{ 1 \} \subseteq G$ is a unit for this fusion law.
	If $G$ is a finite abelian group, this fusion law coincides with the group fusion law introduced in Definition~\ref{def:groupfusionlaw}.
\end{example}

\begin{example}[Representation fusion law]\label{ex:reprfus}
	Let $G$ be a finite group and let $X = \Irr(G)$ be its set of irreducible (complex) characters.
	Then we can define a fusion law on $X$ by declaring
	\[ \chi \in \chi_1 * \chi_2 \iff \chi \text{ is a constituent of } \chi_1 \otimes \chi_2 . \]
	The trivial character is a unit for this fusion law.
\end{example}

\section{Gradings} \label{se:gradings}

This section introduces the necessary preparations for the important connection between axial algebras and groups. On the level of fusion laws, this connection boils down to a morphism from a given fusion law to a group fusion law. We illustrate how to get the strongest possible connection by introducing the finest (abelian) grading of a fusion law.

\begin{definition}
    \begin{enumerate}[\rm (i)]
        \item 
            Let $(X, *)$ be a fusion law and let $(\Gamma, *)$ be a group fusion law.
            A \emph{$\Gamma$-grading} of $(X, *)$ is a morphism $\xi \colon 
            (X, *) \to (\Gamma, *)$. We call the grading \emph{abelian} if 
            $\Gamma$ is an abelian group and we call it \emph{adequate} if 
            $\xi(X)$ generates $\Gamma$.
        \item 
            Every fusion law admits a $\Gamma$-grading where $\Gamma$ is the trivial group; we call this the \emph{trivial} grading.
        \item 
            Let $(X, *)$ be a fusion law. We say that a $\Gamma$-grading $\xi$ of $(X, *)$ is a \emph{finest grading} of $(X, *)$ if every grading of $(X, *)$ factors uniquely through $(\Gamma, *)$, in other words, if for each $\Lambda$-grading $\zeta$ of $(X, *)$, there is a unique group homomorphism $\rho \colon \Gamma \to \Lambda$ such that $\zeta = \rho \circ \xi$.
            (In categorical terms, this can be rephrased as the fact that $\xi$ is an initial object in the category of gradings of $(X, *)$.)
            
            Similarly, we say that an abelian $\Gamma$-grading $\xi$ of $(X, *)$ is a \emph{finest abelian grading} of $(X, *)$ if every abelian grading of $(X, *)$ factors uniquely through $(\Gamma, *)$.
    \end{enumerate}
\end{definition}

\begin{proposition}\label{pr:finest}
    Every fusion law $(X, *)$ admits a unique finest grading, 
    given by the group with presentation
    \[ \Gamma_X := \langle \gamma_x, x \in X \mid \gamma_x \gamma_y = 
    \gamma_z \text{ whenever } z \in x * y \rangle , \]
    with grading map $\xi \colon (X, *) \to (\Gamma_X, *) \colon x 
    \mapsto \gamma_x$. 
    Similarly, there is a unique finest abelian grading, given by the abelianization $\Gamma_X/[\Gamma_X, \Gamma_X]$ of $\Gamma_X$.
    Both gradings are adequate.
\end{proposition}
\begin{proof}
	In order to verify that the map $\xi \colon (X, *) \to (\Gamma_X, *) \colon x \mapsto \gamma_x$ is a morphism of fusion laws,
	we have to check that $\xi(z) \in \xi(x) * \xi(y)$ for all $z \in x * y$.
	This is clear from the definition of $\Gamma_X$, since $\xi(z) = \gamma_z$ and $\xi(x) * \xi(y) = \{ \gamma_x \gamma_y \}$.
	Clearly, $\xi$ is then an adequate grading since $\Gamma_X$ is generated by the elements $\gamma_x$.
	
	Assume now that $\zeta \colon (X, *) \to (\Lambda, *)$ is another grading of $(X, *)$.
	If $x,y,z \in X$ satisfy $z \in x * y$, then $\zeta(z) \in \zeta(x) * \zeta(y) = \{ \zeta(x)\zeta(y) \}$, so the elements
	$\zeta(x)$ satisfy the defining relations of the generators $\gamma_x$ in the presentation for $\Gamma_X$.
	This implies that the map $\rho \colon \Gamma_X \to \Lambda \colon \gamma_x \mapsto \zeta(x)$ is a well defined group homomorphism,
	with $\zeta = \rho \circ \xi$. Since $\xi$ is adequate, the identity $\zeta = \rho \circ \xi$ also uniquely determines the group homomorphism $\rho$.
	
	The proof of the remaining statement is similar.
\end{proof}

\begin{remark}
    There is a lot of ``collapsing'' in the group $\Gamma_X$:
    \begin{enumerate}[\rm (a)]
        \item 
            If $y \in x * y$ for some $y \in X$, then $\gamma_x = 1$ in $\Gamma_X$. In particular, $\gamma_x=1$ for each non-annihilating unit $x \in X$.
        \item 
            All $\gamma_z$, where $z$ runs through some fixed set $x * y$, are equal to each other in~$\Gamma_X$.
        \item
            If $z$ belongs to $x * y$ and to $x * y'$, then $\gamma_y = \gamma_{y'}$.
            Similarly, if $z$ belongs to $x * y$ and to $x' * y$, then $\gamma_x = \gamma_{x'}$.
    \end{enumerate}
\end{remark}

From this it is clear that \(\Gamma_X\) is trivial for most fusion laws $(X, *)$, i.e.,  
they only admit the trivial grading.
We call a fusion law $(X, *)$ \emph{graded} if $\Gamma_X \neq 1$ and \emph{ungraded} otherwise.
It will turn out that graded fusion laws are more interesting for our purposes.

\begin{example} \label{Z/2 examples}
    The Jordan fusion law in Example \ref{ex:jordan} is $\Z/2\Z$-graded. Indeed, 
    the map $\xi \colon X \to \Z/2\Z$ mapping $e$ and $z$ to $0$ and $h$ to $1$ 
    is a fusion law morphism. Notice that this is the finest grading 
    of the Jordan fusion law.

    Similarly, the Ising fusion law in Example \ref{ex:Ising} admits a 
    $\Z/2\Z$-grading: the map $\xi \colon X \to \Z/2\Z$ mapping $e$, $z$ and $q$ to 
    $0$ and $t$ to $1$ is a fusion law morphism. Again, this is the finest 
    grading of the Ising fusion law.
\end{example}

In the remainder of this section, we describe the finest grading of two 
special types of fusion laws: class fusion laws and representation fusion laws.

The class fusion law of a group $G$ was introduced in Example 
\ref{ex:classlaw}. For $g\in G$, let $\bar g$ denote the image of $g$ in 
$G/[G,G]$. 

\begin{proposition}\label{pr:fus}
	Let $(X, *)$ be the class fusion law of a group $G$. Then the finest 
	grading of $(X, *)$ is given by the group $\Gamma = G/[G,G]$ with 
	grading map $X \to \Gamma \colon \conj{G}{g} 
	\mapsto \overline{g}$.
\end{proposition}
\begin{proof}
	By definition, the finest grading of $(X, *)$ is the group \[ 
	\Gamma_X := \langle \gamma_C, C \in X \mid \gamma_C \gamma_D = \gamma_E 
	\text{ whenever } CD \cap E \neq \emptyset \rangle . \] Consider the map 
	$\varphi \colon G \to \Gamma_X \colon g \mapsto \gamma_{(\conj{G}{g})}$ and 
	notice that $\varphi$ is a group morphism, precisely by the defining relations 
	of $\Gamma_X$.  It is clearly surjective; moreover, $\varphi(\conj{g}{h}) = 
	\varphi(h)$ for all $g,h \in G$.  It follows that for each commutator $[g,h] = 
	g h g^{-1} h^{-1}$, we have $\varphi([g,h]) = \varphi(\conj{g}{h}) 
	\varphi(h)^{-1} = 1$; hence $[G,G] \leq \ker \varphi$.  Hence $\varphi$ 
	induces a group epimorphism $\tilde\varphi \colon \Gamma \to \Gamma_X$.
	
	Finally, the map $\Gamma_X \to \Gamma \colon \gamma_{(\conj{G}{g})} \to 
	g[G,G]$ is well defined because it kills each relator of $\Gamma_X$, and this 
	map provides an inverse of $\tilde\varphi$, showing that it is an isomorphism 
	from $\Gamma_X$ to $\Gamma$. 
\end{proof}

Recall the definition of the representation fusion law from Example 
\ref{ex:reprfus}.

\begin{proposition}\label{pr:rep}
    Let $G$ be a finite group and let $(X,\ast)$ be the representation fusion 
    law of $G$. Then the finest grading of $(X,\ast)$ is given by 
    $\Gamma_X=Z(G)^*=\Irr(Z(G))$ with grading map $X \to \Irr(Z(G))\colon 
    \chi\mapsto \frac{\chi_{Z(G)}}{\chi(1)}$.
\end{proposition}
\begin{proof}[\stepcounter{footnote}Proof\,${}^\thefootnote$]
    \footnotetext{Thanks to David Craven and Frieder Ladisch for providing the central argument 
    in this proof. As Frieder Ladisch pointed out to us, this result also follows 
    from \cite{GN}*{Example~3.2 and Corollary~3.7}.}
    Consider an arbitrary adequate grading $f\colon X\to\Gamma$ and define \[K = 
    \left\{ \chi\in \Irr(G) \mid f(\chi)=1 \right\}.\] Let $H = 
    \bigcap_{\chi \in K} \ker \chi$. If \(\chi\in K\) then it is clear that
    \(H\leq \ker\chi\); we aim to show the opposite inclusion. Consider 
    $\theta = \sum_{\chi\in K} 
    \chi$, which may be considered as a character of $G/H$. Since $\theta$ 
    is faithful as a character of $G/H$, by the Burnside--Brauer theorem 
    every irreducible character of $G/H$ is a constituent of some power 
    of $\theta$. Now since $f$ is trivial on each constituent of $\theta$, 
    it also is trivial on all irreducible characters of $G/H$. Thus if 
    \(H\leq \ker\chi\) then \(\chi\in K\). We have now established that
    $K=\{ \chi \in \Irr(G) \mid H \leq \ker\chi \}$. 

    Note that $f(\bar\chi)=f(\chi)^{-1}$. Indeed, $\mathds{1}_G$ is a 
    constituent of $\chi\bar\chi$ ; that is, $\mathds{1}_G\in \chi*\bar\chi$. 
    This means that $f(\chi)f(\bar\chi)=f(\mathds{1}_G)=1$.
    
    Now let $\psi\in\Irr(H)$ and let $\chi$ and $\eta$ be constituents of the 
    induced character $\psi^G$, so that \(\psi\) is a constituent of the 
    restrictions \({\chi}_{{}_H}\) and \({\eta}_{{}_H}\) by Frobenius 
    reciprocity. Thus 
    \(0<\langle\eta_{{}_H},\chi_{{}_H}\rangle=\langle\mathds{1}_H, 
    (\chi\overline\eta)_{{}_H}\rangle\) (where \(\langle-,-\rangle\) represents 
    the inner product of class functions) and hence $\mathds{1}_H$ is a 
    constituent of $(\chi\overline\eta)_{{}_H}$.  Since $H\unlhd G$, a corollary 
    of Clifford's theorem now implies that $\chi\overline\eta$ has a constituent 
    $\theta\in\Irr(G)$ with $H\leq\ker\theta$ (see for example 
    \cite{isaacs:1976}*{Corollary~6.7}). Hence $f(\chi)f(\eta)^{-1}= 
    f(\chi\overline\eta)=f(\theta)=1$. That is, $f(\chi)=f(\eta)$.
    Thus, we obtain a well-defined map $f'\colon\Irr(H)\to \Gamma$ by setting 
    $f'(\psi)=f(\chi)$ for any constituent $\chi$ of $\psi^G$.
  
    Next, we show that $H$ is in the center of $G$, so let us assume that
    there is some non-central $x\in H$. As $x$ is not central, the column 
    orthogonality relations imply that there must be a character $\chi\in 
    \Irr(G)$ such that $\lvert\chi(x)\rvert \leq \chi(1)$ and, therefore, 
    there is a constituent $\theta$ of $\chi\overline\chi$ with 
    $\theta(x) \neq \theta(1)$. On the other hand, $f(\theta) = 
    f(\chi)f(\overline\chi) = 1$, yielding $\theta \in K$. This means that  
    $H\leq\ker\theta$ and so $\theta(x)=\theta(1)$; a contradiction.  
    
    Since $H$ is central, the map $X\to\Irr(H) \colon \chi\mapsto 
    \frac{\chi_H}{\chi(1)}$ is defined and $f$ is the composition of this 
    map and $f'$. Clearly, the map $X\to\Irr(H)$ factors through the 
    similar map $X\to\Irr(Z(G))$, and so the claim of the proposition holds.
\end{proof}

\begin{remark}
	\begin{enumerate}[(i)]
		\item It is immediate from the definition that the finest 
		grading of the union of fusion laws $(X,*)$ and $(Y,*)$ is the 
		free product of $\Gamma_X$ and $\Gamma_Y$ with the obvious grading map.
		\item The similar question about the finest grading of the 
		product $(X\times Y,*)$ is more difficult. It is easy to see that there 
		is a grading of $(X\times Y,*)$ by the group $\Gamma_X\times\Gamma_Y$. 
		However, it is equally easy to find examples where this is not the 
		finest grading. For instance, if $(X, *)$ is an empty fusion law (i.e., $x_1 * x_2 = \emptyset$ for all $x_1,x_2 \in X$)
		and $(Y, *)$ is any fusion law, then the product $(X \times Y, *)$ is again an empty fusion law,
		but the finest grading of an empty fusion law is always a free group.
	\end{enumerate}
\end{remark}

\section{Decomposition algebras} \label{se:decalg}

We are now ready to introduce decomposition algebras. We believe that they provide the right axiomatic framework to study all algebras reminiscent of axial algebras. It is the first definition of such algebras that allows for an interesting definition of homomorphisms. For each choice of a base ring and a fusion law, this will give rise to a corresponding \emph{category of decomposition algebras}.
We refer to Appendix~\ref{se:appendix} for further categorical properties.

\begin{definition}
    Let $R$ be a commutative ring and let $\Phi = (X, *)$ be a fusion law.
    \begin{enumerate}[(i)]
        \item
          A \emph{$\Phi$-decomposition} of an $R$-algebra $A$ (not assumed to be 
          commutative, associative or unital) is a direct sum decomposition $A = 
          \bigoplus_{x \in X} A_x$ (as $R$-modules) such that $A_x A_y \subseteq 
          A_{x * y}$ for all $x,y \in X$, where $A_{Y} := \bigoplus_{y \in Y} 
          A_y$ for all $Y \subseteq X$.
        \item
            A \emph{$\Phi$-decomposition algebra} is a triple $(A, \I, \Omega)$ 
            where $A$ is an $R$-algebra,
            $\I$ is an index set and $\Omega$ is a tuple%
            \footnote{Formally, we could define $\Omega$ as a set and define 
            this ``tuple'' as a map from $\I$ to $\Omega$, but we will
            not do so in order not to make our notation unnecessarily 
            complicated.}
            of $\Phi$-decompositions of $A$ indexed by $\I$.
            We will usually write the corresponding decompositions as $A = 
            \bigoplus_{x \in X} A_x^i$, so
            \[ \Omega = \bigl( ( A_x^i )_{x \in X} \mid i \in \I \bigr) ; \]
            we sometimes use the shorthand notation $\Omega[i] := ( A_x^i )_{x 
            \in X}$.
            Notice that we do not require the decompositions to be distinct.
    \end{enumerate}
    We will often omit the explicit reference to $\Phi$ if it is clear from the 
context and simply talk about decompositions and decomposition algebras.
\end{definition}

\begin{example}
    Consider the following fusion law $(X, *)$ on $X = \{ e,z \}$:
	\begin{center}
		\renewcommand{\arraystretch}{1.2}
		\setlength{\tabcolsep}{0.75em}
		\begin{tabular}[b]{c|cc}
			$*$	& $e$ & $z$ \\
			\hline
			$e$ & $\{ e \}$ & $\emptyset$ \\
			$z$ & $\emptyset$ & $\{ z \}$ \\
		\end{tabular}
	\end{center}
	Let $A$ be any \emph{commutative associative} algebra over a commutative ring $R$.
	Let $\{ a_i \mid i \in \I \} \subseteq A$ be any collection of idempotents in $A$, indexed by some set $\I$.
	For each $i \in \I$, the algebra $A$ decomposes as $A = a_i A \oplus (1-a_i) A$.
	Write $A_e^i := a_i A$ and $A_z^i := (1-a_i) A$.
	Then each decomposition $A = A_e^i \oplus A_z^i$ is indeed an $(X, *)$\dash decomposition.
	If we write $\Omega$ for the $\I$-tuple of all those decompositions,
	then $(A, \I, \Omega)$ is a decomposition algebra.
\end{example}

\begin{example}
    Consider the Jordan fusion law $(X, *)$ from Example~\ref{ex:jordan}.
	Let $J$ be any \emph{Jordan} algebra over a commutative ring $R$.
	Let $\{ a_i \mid i \in \I \} \subseteq J$ be any collection of idempotents in $J$, indexed by some set $\I$.
	For each $i \in \I$, the algebra $J$ admits a \emph{Peirce decomposition} into the \emph{Peirce subspaces} with respect to the idempotent $a_i$
	(see, e.g., \cite[Chapter III]{Jacobson}):
	\[ J = J^i_0 \oplus J^i_1 \oplus J^i_{1/2}. \]
	By \cite[Chapter III, \S 1, Lemma 1]{Jacobson}, each of those decompositions is indeed an $(X, *)$\dash decomposition
	(where $e$ corresponds to $1$, $z$ to $0$ and $h$ to $1/2$).
	If we write $\Omega$ for the $\I$-tuple of all those decompositions,
	then $(A, \I, \Omega)$ is a decomposition algebra.
\end{example}

\begin{remark}
    Let $\Phi = (X, *)$ be a fusion law and let $(A, \I, \Omega)$ be a $\Phi$\dash decomposition algebra.
    If $e \in X$ is annihilating (see Definition~\ref{def:fus-unit}), then each subspace $A^i_e$ is annihilating for the algebra $A$, in the sense that $A^i_e \cdot A = 0$.
    Similarly, if $e \in X$ is absorbing (see Definition~\ref{def:fus-unit}), then each $A^i_e$ is an ideal: $A^i_e \cdot A \subseteq A^i_e$.
\end{remark}

The decomposition algebras with respect to a fixed fusion law form a nice category.
\begin{definition}\label{def:DEC}
    Let $R$ be a commutative ring and let $\Phi = (X, *)$ be a fusion law.
    We define a category $\DEC$ having as objects the $\Phi$\dash decomposition algebras over~$R$.
    If $(A, \I, \Omega_A)$ and $(B, \J, \Omega_B)$ are two objects, with
    \[ \Omega_A = \bigl( ( A_x^i )_{x \in X} \mid i \in \I \bigr) , \qquad
        \Omega_B = \bigl( ( B_x^j )_{x \in X} \mid j \in \J \bigr) , \]
    then the \emph{morphisms} between $(A, \I, \Omega_A)$ and $(B, \J, \Omega_B)$ are defined to be pairs $(\varphi, \psi)$
    where $\varphi \colon A \to B$ is an $R$-algebra morphism and $\psi \colon \I \to \J$ is a map (of sets)
    such that
    \[ \varphi(A_x^i) \subseteq B_x^{\psi(i)} \]
    for all $x \in X$ and all $i \in \I$.
\end{definition}

\begin{proposition}\label{pr:view}
    If $\xi \colon (X, *) \to (Y, *)$ is a fusion law morphism and $(A, \I, \Omega)$ is an $(X, *)$\dash decomposition algebra,
    then $A$ can also be viewed as a $(Y, *)$\dash decomposition algebra $(A, \I, \Sigma)$ by declaring
    \[ A_y^i := A_{\xi^{-1}(y)}^i = \bigoplus_{x \in \xi^{-1}(y)} A_x^i \]
    for each $i \in \I$ and each $y \in Y$.
    This induces a functor
    \[ F_\xi \colon (X, *)\mhyphen\mathbf{Dec}_R \to (Y, *)\mhyphen\mathbf{Dec}_R . \]
\end{proposition}
\begin{proof}
    We have to verify that for all $y,z \in Y$, we have $A_y^i A_z^i \subseteq A_{y * z}^i$.
    By the definition of a fusion law morphism, we have
    \[ \xi^{-1}(y) * \xi^{-1}(z) \subseteq \xi^{-1}(y * z) , \]
    and hence indeed
    \begin{align*}
        A_y^i A_z^i
        &= A_{\xi^{-1}(y)}^i A_{\xi^{-1}(z)}^i \\
        &\subseteq A_{\xi^{-1}(y) * \xi^{-1}(z)}^i \\
        &\subseteq A_{\xi^{-1}(y * z)}^i = A_{y * z}^i,
    \end{align*}
    proving the proposition.
\end{proof}

In Appendix~\ref{se:appendix}, we study the category $\DEC$ in some more detail.

\section{Axial decomposition algebras}\label{se:axial}

In this section, we explain how axial algebras fit into the framework of decomposition algebras.

\begin{definition}
    Let $\Phi = (X, *)$ be a fusion law with a distinguished unit $e \in X$.
    For each $x \in X$, let $\lambda_x \in R$.
    A $\Phi$-decomposition algebra $(A, \I, \Omega)$ will be called \emph{left-axial} (with \emph{parameters} $\lambda_x$)
    if for each $i \in \I$, there is some non-zero $a_i \in A_e^i$ (called a \emph{left axis}) such that:
    \begin{equation}\label{eq:left-axial}
        a_i \cdot b = \lambda_x b \quad \text{for all } x \in X \text{ and for all } b \in A_x^i .
    \end{equation}
    Similarly, $(A, \I, \Omega)$ is a \emph{right-axial} decomposition algebra (with \emph{parameters} $\lambda_x$)
    if for each $i \in \I$, there is some non-zero $a_i \in A_e^i$ (called a \emph{right axis}) such that:
    \begin{equation}\label{eq:right-axial}
        b \cdot a_i = \lambda_x b \quad \text{for all } x \in X \text{ and for all } b \in A_x^i .
    \end{equation}
    Of course, if $A$ is commutative, then we drop the prefix ``left'' or ``right'' and simply talk about axial decomposition algebras.
    We call a (left- or right-)axial decomposition algebra \emph{primitive} if $A_e^i = Ra_i$ for each $i \in \I$.
\end{definition}

\begin{remark}
    Recall from~\cite{HRS} that an \emph{axial algebra} is a commutative algebra $A$ generated by a set $E$ of idempotents (called \emph{axes}),
    such that for each axis $c \in E$, the left multiplication operator $\ad_c \colon A \to A \colon x \mapsto cx$ is semi-simple and
    its eigenspaces multiply according to a given fusion law $\Phi = (X, *)$ with $X \subseteq R$.

    Every axial algebra is an axial decomposition algebra.
    Indeed, if $(A, E)$ is an axial algebra, then for each $c \in E$, there is a corresponding decomposition $A = \bigoplus_{x \in X} A_x^c$,
    so certainly $(A, E, \Omega)$ with $\Omega = \bigl\{ ( A_x^c )_{x \in X} \mid c \in E \bigr\}$ is a decomposition algebra.
    It is indeed axial, with $a_c = c$ for each $c \in E \subseteq A$ and $\lambda_x = x$ for each $x \in X \subseteq R$.

    On the other hand, axial decomposition algebras are more general objects than axial algebras, in four ways:
    \begin{itemize}
        \item
            The elements $a_c \in A$ are not required to be idempotents.
            If the corresponding parameter $\lambda_e \neq 0$ is a unit in $R$ (for example when $R$ is a field), then we can rescale $a_c$ to an idempotent.
            If $\lambda_e = 0$, then $a_c^2 = 0$, i.e., $a_c$ is nilpotent.
        \item
            The algebra $A$ is not assumed to be generated by the axes.
        \item
            By distinguishing between $x \in X$ and $\lambda_x \in R$, we allow the possibility that some of the $\lambda_x \in R$ coincide.
        \item
            The algebra $A$ is not assumed to be commutative.
    \end{itemize}
\end{remark}

We now make the class of (left) axial decomposition algebras into a category.
\begin{definition}
    Let $\Phi = (X, *)$ be a fusion law with a distinguished unit $e \in X$
    and let $\lambda \colon X \to R \colon x \mapsto \lambda_x$ be an arbitrary map, called the \emph{evaluation map}.
    We define a category $\AXDEC$ with as objects the axial $\Phi$-decomposition algebras together with the collection of left axes,
    for the choice of parameters $\lambda_x$ given by the evaluation map.
    In other words, the objects are quadruples $(A, \I, \Omega, \alpha)$, where $(A, \I, \Omega)$ is a $\Phi$-decomposition algebra
    and $\alpha \colon \I \to A \colon i \mapsto a_i$ is a map such that $a_i \in A_e^i$ and \eqref{eq:left-axial} holds.

    The morphisms in this category are the morphisms $(\varphi, \psi) \colon (A, \I, \Omega_A, \alpha) \to (B, \J, \Omega_B, \beta)$
    of decomposition algebras such that $\varphi \circ \alpha = \beta \circ \psi$, i.e.,
    $\varphi$ maps each axis $a_i$ to the corresponding axis $b_{\psi(i)}$.
\end{definition}

\section{The (universal) Miyamoto groups}\label{se:Miy}

Let $\Gamma$ be a finite group fusion law.
To each $\Gamma$-decomposition algebra $(A, \I, \Omega)$, we will associate a subgroup of the automorphism group of $A$,
called the Miyamoto group of $(A, \I, \Omega)$.
We will also construct a cover of this group, which we call the universal Miyamoto group and which has nicer functorial properties than
the Miyamoto group itself.

We will, at the same time, construct subgroups of these Miyamoto groups, one for 
each subgroup of the character group.
\begin{definition}
    Let $R^\times$ be the group of invertible elements of the base ring $R$.
    An \textit{$R$-character} of $\Gamma$ is a group homomorphism $\chi \colon 
    \Gamma \to R^\times$.
    The \textit{$R$-character group} of $\Gamma$ is the group $\X_R(\Gamma)$ 
    consisting of all $R$-characters of $\Gamma$,
    with group operation induced by multiplication in $R^\times$.
    When the base ring $R$ is clear from the context, we will sometimes omit it 
and simply talk about characters and the character group.
\end{definition}
Notice that depending on $R$, the group $\X_R(\Gamma)$ might be infinite even if 
$\Gamma$ is finite.

\begin{definition}\label{def:Miy}
    Let $(A, \I, \Omega)$ be a $\Gamma$-decomposition algebra.
    \begin{enumerate}[(i)]
        \item
            Let $\chi \in \X_R(\Gamma)$.
            For each decomposition $(A^i_g)_{g \in \Gamma} \in \Omega$, we define a linear map
            \[ \tau_{i,\chi} \colon A \to A \colon a \mapsto \chi(g) a \quad \text{for all } a \in A^i_g ; \]
            we call this a \emph{Miyamoto map}.
            It follows immediately from the definitions that each $\tau_{i,\chi}$ is an automorphism of the $R$-algebra $A$.
            Notice that each $\tau_{i,\chi}$ has finite order (dividing the order of $\chi$ in $\X_R(\Gamma)$).
        \item
            Let $\Y$ be any subgroup of the character group $\X_R(\Gamma)$.
            We then define the \emph{Miyamoto group} with respect to $\Y$ as
            \[ \Miy_\Y(A, \I, \Omega) := \langle \tau_{i,\chi} \mid i \in \I, \chi \in \Y \rangle \leq \Aut(A) . \]
            Two important special cases get their own notation:
            \begin{align*}
                \Miy(A, \I, \Omega) &:= \Miy_{\X_R(\Gamma)}(A, \I, \Omega); \\
                \Miy_\chi(A, \I, \Omega) &:= \Miy_{\langle \chi \rangle}(A, \I, \Omega) \quad \text{for a given character } \chi \in \X_R(G).
            \end{align*}
        \item\label{Miy:closed}
            We call $(A, \I, \Omega)$ \emph{Miyamoto-closed} with respect to $\Y$
            if the set $\Omega$ is invariant under the Miyamoto group with 
            respect to $\Y$.
            That is for each $i \in \I$ and each $\chi \in \Y$, there is a permutation%
            \footnote{In the situation where some of the decompositions $(A^j_g)_{g \in \Gamma} \in \Omega$ coincide, there might be some freedom
            in the choice of the permutation $\pi_{i,\chi}$, but this choice will be irrelevant for us.}
            $\pi_{i,\chi}$ of $\I$ such that $\tau_{i,\chi}$ maps each decomposition $(A^j_g)_{g \in \Gamma} \in \Omega$
            to the decomposition $(A^{\pi_{i,\chi}(j)}_g)_{g \in \Gamma} \in \Omega$.
            Notice that in this case, each pair $(\tau_{i,\chi}, \pi_{i,\chi})$ is an automorphism of $(A, \I, \Omega)$ in the category $\GDEC$.
            In particular, the conjugate of a Miyamoto map by a Miyamoto map is again a Miyamoto map.
    \end{enumerate}
\end{definition}

\begin{example}\label{ex:Miy-DMVC}
    The simplest non-trivial example is the case where $\Gamma = \Z/2\Z$ and $\Y = 
    \{ 1, \chi \}$ where $\chi$ maps the non-trivial element of $\Gamma$ to $-1 
    \in R$ (assuming that $-1 \neq 1$ in $R$).
    In the case of axial algebras, we recover the definition of the Miyamoto 
group as in \cite{DMVC}*{Definition 2.5}.
\end{example}
The Miyamoto group is interesting---it is a subgroup of the automorphism group 
of the algebra---but is not so easy to control (cf.  
Example~\ref{ex:Miy-not-functorial} below). It is useful to construct a cover of 
this group, which we call the \emph{universal Miyamoto group}.


\begin{definition}
    We keep the notations from Definition~\ref{def:Miy}
    and assume that $(A, \I, \Omega)$ is Miyamoto-closed with respect to $\Y$.
    Recall our convention from Notation~\ref{not:conv}.
    We define the \emph{universal Miyamoto group} with respect to $\Y$ as the group given by the following presentation.
    For each $i \in I$, we let $\Y_i$ be a copy of the group $\Y$ and we denote its elements by
    \[ \Y_i = \{ t_{i,\chi} \mid \chi \in \Y \} . \]
    For each $a = t_{i,\chi} \in \Y_i$, we write $\overline{a}$ for the corresponding Miyamoto map $\tau_{i,\chi} \in \Miy(A, \I, \Omega)$.
    Notice that for each $i \in \I$, the group
    \[ \overline{\Y_i} := \{ \overline{a} \mid a \in \Y_i \} =  \{ \tau_{i,\chi} \mid \chi \in \Y \} \]
    is an \textit{abelian} subgroup of $\Miy_\Y(A, \I, \Omega)$.
    
    We will define the universal Miyamoto group $\UMiy_{\Y}(A, \I, \Omega)$ as a quotient of the free product $\freeprod_{i \in \I} \Y_i$
    by conjugation relations between the groups~$\Y_i$ that exist ``globally'' between the corresponding groups $\overline{\Y_i}$ in $\Miy(A, \I, \Omega)$.
    More precisely, let $\U := \bigcup_{i \in \I} \Y_i$; for each $a \in \U$, we consider the set
    \begin{equation}\label{eq:Ra}
        R_{\overline a} := \{ (j,k) \in \I \times \I \mid \conj{\overline{a}}{\tau_{j,\chi}} = \tau_{k,\chi} \text{ for all } \chi \in \Y \} .
    \end{equation}
    We then let
    \begin{multline*}
        \UMiy_{\Y}(A, \I, \Omega) \\
            := \left\langle \freeprod_{i \in \I} \Y_i \Bigm\vert \conj{a}{t_{j,\chi}} = t_{k,\chi} \text{ for all }
            a \in \U \text{, all } (j,k) \in R_{\overline a}  \text{ and all } \chi \in \Y \right\rangle.
    \end{multline*}
\end{definition}

\begin{remark}
        The reader might wonder why we only consider conjugation relations that exist globally and do not define the universal Miyamoto group as the group
        \[ \left\langle \freeprod_{i \in \I} \Y_i \Bigm\vert \conj{a}{b} = c \text{ for all }
                a,b,c \in \U \text{ satisfying } \conj{\overline{a}}{\overline{b}} = \overline{c} \right\rangle.
        \]
        instead.
        The problem with this definition is that some conjugation relations might hold ``by coincidence'' and we do not want to transfer those to the universal Miyamoto group. For instance, Theorem~\ref{thm:axial-functor} below would become false with this seemingly easier definition.
\end{remark}

On the other hand, since $(A, \I, \Omega)$ is Miyamoto-closed with respect to $\Y$, we always have many conjugation relations at our disposal.

\begin{lemma}\label{le:umiy}
    Let $i,j \in \I$.
    \begin{enumerate}[\rm (i)]
        \item\label{umiy:id}
        For each $\chi,\chi' \in \Y$, the relation
        \[ \conj{t_{i,\chi}}{t_{j,\chi'}} = t_{\pi_{i,\chi}(j),\chi'} \]
        holds in $\UMiy_\Y(A, \I, \Omega)$.
        \item\label{umiy:eq}
        If $\tau_{i,\chi} = \tau_{j,\chi}$ for all $\chi \in \Y$, then also $t_{i,\chi} = t_{j,\chi}$ for all $\chi \in \Y$.
    \end{enumerate}
\end{lemma}
\begin{proof}
    \begin{enumerate}[\rm (i)]
        \item
        Let $a = t_{i,\chi}$ for some $\chi \in \Y$.
        Since $(A, \I, \Omega)$ is Miyamoto-closed with respect to $\Y$, we have
        $\conj{\tau_{i,\chi}}{\tau_{j,\chi'}} = \tau_{\pi_{i,\chi}(j),\chi'}$ for all $\chi' \in \Y$
        and therefore $(j, \pi_{i,\chi}(j)) \in R_{\overline a}$.
        It follows that all relations of the form
        \[ \conj{t_{i,\chi}}{t_{j,\chi'}} = t_{\pi_{i,\chi}(j),\chi'} \]
        hold in $\UMiy_\Y(A, \I, \Omega)$.
        \item
        Let $a = t_{i,\chi}$ for some $\chi \in \Y$.
        Recall that $\overline{\Y_i}$ is abelian, hence $\tau_{i,\chi}$ commutes with $\tau_{i,\chi'}$ for all $\chi' \in \Y$.
        Since $\tau_{i,\chi'} = \tau_{j,\chi'}$, it follows that $(i, j) \in R_{\overline a}$.
        Therefore, the relations
        \[ \conj{t_{i,\chi}}{t_{i,\chi'}} = t_{j, \chi'} \]
        hold in $\UMiy_\Y(A, \I, \Omega)$.
        Since $t_{i,\chi}$ and $t_{i,\chi'}$ both belong to the abelian group $\Y_i \leq \UMiy_\Y(A, \I, \Omega)$, we conclude that the relation
        $t_{i,\chi'} = t_{j, \chi'}$ holds in $\UMiy_\Y(A, \I, \Omega)$.
        \qedhere
    \end{enumerate}
\end{proof}

\begin{proposition}\label{pr:central}
    Let $\Y \leq \X_R(\Gamma)$ and let $(A, \I, \Omega)$ be Miyamoto-closed with respect to $\Y$.
    Then $\UMiy_\Y(A, \I, \Omega)$ is a central extension of $\Miy_\Y(A, \I, \Omega)$.
\end{proposition}
\begin{proof}
    Let $\widehat G := \UMiy_{\Y}(A,\I,\Omega)$, $G := \Miy_{\Y}(A,\I,\Omega)$ and
    $\U := \bigcup_{i \in \I} \Y_i \subseteq \widehat G$; then $\widehat G = \langle \U \rangle$.
    It is immediately clear from the definition of $\UMiy_{\Y}(A, \I, \Omega)$ that the map
    $\U \to G \colon a \mapsto \overline{a}$ extends to an epimorphism $\Phi \colon \widehat G \to G$;
    it remains to show that $\ker\Phi$ is central.
    
    Let $z \in \ker\Phi$ be arbitrary; as each generator $a \in \U$ has finite order, we can write $z = a_m \dotsm a_1$ with $a_i \in \U$.
    We have to show that $\conj{z}{b} = b$ for each $b = t_{j, \chi'} \in \U$.
    Fix such an element $b \in \U$.
    For each $k \in \{ 0,\dots,m \}$, we write
    \begin{align*}
        b_k &:= \conj{a_k \dotsm a_1}{b} \in \widehat G \quad \text{and} \\
        c_k &:= \conj{\overline a_k \dotsm \overline a_1}{\overline b} \in G. 
    \end{align*}
%
    By repeatedly applying Lemma~\ref{le:umiy}\eqref{umiy:id}, we see that each $b_k$ is again of the form $t_{j_k, \chi'}$ for some $j_k \in \I$
    (which only depends on $z$ and $j$ but not on $\chi'$)
    and that $c_k = \overline{b_k}$ for each $k$.
    
    In particular, $\overline{b_m} = c_m = \conj{\Phi(z)}{\overline{b}} = \overline{b}$ with $b = t_{j,\chi'}$ and $b_m = t_{j_m,\chi'}$.
    Hence $\tau_{j_m, \chi'} = \tau_{j, \chi'}$.
    Because this holds for all $\chi' \in \Y$, Lemma~\ref{le:umiy}\eqref{umiy:eq} now implies that $t_{j_m, \chi'} = t_{j, \chi'}$ for all $\chi'$.
    Varying $j \in \I$ finishes the proof.
\end{proof}

For \textit{surjective} morphisms between decomposition algebras, both $\Miy_\Y$ and $\UMiy_\Y$ are functorial.
The following easy lemma is the key point.
\begin{lemma}\label{le:tau_i}
    Let $(\varphi, \psi)$ be a morphism between two $\Gamma$-decomposition algebras $(A, \I, \Omega_A)$ and $(B, \J, \Omega_B)$.
    Then for each $i \in \I$ and $\chi \in \X_R(\Gamma)$, we have $\varphi \circ \tau_{i,\chi} = \tau_{\psi(i),\chi} \circ \varphi$.
\end{lemma}
\begin{proof}
    Let $a \in A_g^i$ for some $g \in \Gamma$.
    Then on the one hand, $\varphi(\tau_{i,\chi}(a)) = \varphi(\chi(g)a) = \chi(g) \varphi(a)$, while on the other hand,
    $\varphi(a) \in B_g^{\psi(i)}$ and hence
    $\tau_{\psi(i),\chi}(\varphi(a)) = \chi(g) \varphi(a)$ as well.
    Since $A = \bigoplus_{g \in \Gamma} A_g^i$, the result follows.
\end{proof}
\begin{proposition} \label{prop:surj-functor}
    Let $\Y \leq \X_R(\Gamma)$.
    Let $(\varphi, \psi)$ be a morphism between two $\Gamma$\dash decomposition algebras $(A, \I, \Omega_A)$ and $(B, \J, \Omega_B)$.
    Assume that $\varphi$ is surjective.
    Then:
    \begin{enumerate}[\rm (i)]
        \item 
            There is a corresponding morphism $\theta \colon \Miy_\Y(A, \I, \Omega_A) \to \Miy_\Y(B, \J, \Omega_B)$
            mapping each generator $\tau_{i,\chi}$ of $\Miy_\Y(A, \I, \Omega_A)$ to the corresponding generator
            $\tau_{\psi(i),\chi}$ of $\Miy_\Y(B, \J, \Omega_B)$.
        \item 
            There is a corresponding morphism $\widehat\theta \colon \UMiy_\Y(A, \I, \Omega_A) \to \UMiy_\Y(B, \J, \Omega_B)$
            mapping each generator $t_{i,\chi}$ of $\UMiy_\Y(A, \I, \Omega_A)$ to the corresponding generator
            $t_{\psi(i),\chi}$ of $\UMiy_\Y(B, \J, \Omega_B)$.
    \end{enumerate}
\end{proposition}
\begin{proof}
    \begin{enumerate}[\rm (i)]
        \item
            It suffices to verify that if the $\tau_{i,\chi}$ satisfy some relation
            \[ \tau_{i_1, \chi_1} \dotsm \tau_{i_\ell, \chi_\ell} = 1 \]
            inside $\Aut(A)$, then also
            \[ \tau_{\psi(i_1), \chi_1} \dotsm \tau_{\psi(i_\ell), \chi_\ell} = 1 \]
            inside $\Aut(B)$.
            This follows immediately from Lemma~\ref{le:tau_i} and the fact that $\varphi$ is surjective.
        \item
            We have to show that each relator of $\UMiy_\Y(A, \I, \Omega_A)$ is killed by~$\widehat\theta$.
            Consider a relator
            \[ r = t_{k,\chi'}^{-1} \cdot \conj{t_{i,\chi}}{t_{j,\chi'}} \quad \text{with} \quad (j, k) \in R_{\tau_{i,\chi}} . \]
            Then by definition, we have
            $\tau_{k,\chi'} = \conj{\tau_{i,\chi}}{\tau_{j,\chi'}}$ in $\Miy_\Y(A, \I, \Omega_A)$.
            By Lemma~\ref{le:tau_i}, this implies that
            $\tau_{\psi(k),\chi'} \circ \varphi = \conj{\tau_{\psi(i),\chi}}{\tau_{\psi(j),\chi'}} \circ \varphi$.
            Since $\varphi$ is surjective, it follows that
            $\tau_{\psi(k),\chi'} = \conj{\tau_{\psi(i),\chi}}{\tau_{\psi(j),\chi'}}$ in $\Miy_\Y(B, \J, \Omega_B)$.
            Because this holds for all \mbox{$\chi' \in \Y$}, we have
            \[ (\psi(j), \psi(k)) \in R_{\tau_{\psi(i),\chi}} . \]
            Now $\widehat\theta$ maps the given relator $r$ to 
            $t_{\psi(k),\chi'}^{-1} \cdot 
            \conj{t_{\psi(i),\chi}}{t_{\psi(j),\chi'}}$, and by the definition of 
            $\UMiy_\Y(B, \J, \Omega_B)$, this element is trivial.
        \qedhere
    \end{enumerate}
\end{proof}
The requirement that $\varphi$ is surjective cannot be dropped in general, as 
the following generic type of example illustrates.
\begin{example}
    Let $\Gamma = \{ 1, \sigma \}$ be the group of order $2$ and let $\Y = \{ 1 , \chi \}$ as in Example~\ref{ex:Miy-DMVC} above.
    Since there is only one non-trivial character in $\Y$, we will omit it from our notation and, for example, write $\tau_i$ in place of $\tau_{i,\chi}$.
    Let $(A, I, \Omega)$ be a $\Gamma$\dash decomposition algebra.
    The only (very weak) assumption we make, is the existence of three different $j,k,\ell \in I$ such that there is a relation
    $\conj{\tau_k}{\tau_j} = \tau_\ell$.

    We will now construct another $\Gamma$-decomposition algebra $(B, J, \Omega')$ and a morphism $(\varphi, \psi) \colon (A, I, \Omega) \to (B, J, \Omega')$
    such that the map $t_i \mapsto t_{\psi(i)}$ does \emph{not} induce a group morphism between the corresponding universal Miyamoto groups.

    Let $B = A \oplus M$, where $M$ is a free $R$-module of rank $2$ with basis $\{ e, f \}$, and extend the multiplication of $A$ to $B$ trivially
    ($AM = MA = 0$).
    Let $\varphi \colon A \to B$ be the natural inclusion.
    Let $J = I \times \{ 1, 2 \}$; we will construct two decompositions of $B$ for each decomposition of $A$ in $\Omega$.
    Define
    \begin{align*}
        \Omega'[i,1] &:= (A^i_1 \oplus Re, A^i_\sigma \oplus Rf) \quad \text{and} \\
        \Omega'[i,2] &:= (A^i_1 \oplus Rf, A^i_\sigma \oplus Re) .
    \end{align*}
    If we arbitrarily choose $c_i \in \{ 1,2 \}$ for each $i \in I$, then the map $\psi \colon I \to J \colon i \mapsto (i, c_i)$
    will give rise to a morphism $(\varphi, \psi)$ of $\Gamma$-decomposition algebras.
    In particular, this holds if we choose $c_j = c_k = 1$ and $c_\ell = 2$.
    Now consider the corresponding Miyamoto involutions $\tau_{(j,1)}$, $\tau_{(k,1)}$ and $\tau_{(\ell, 2)}$ of $B$;
    then $\tau_{(j,1)}$ and $\tau_{(k,1)}$ fix the element $e$ whereas $\tau_{(\ell, 2)}$ maps $e$ to $-e$.
    In particular,
    \[ \conj{\tau_{\psi(k)}}{\tau_{\psi(j)}} = \conj{\tau_{(k,1)}}{\tau_{(j,1)}} \neq \tau_{(\ell,2)} = \tau_{\psi(\ell)} . \]
    Hence the map $t_{i} \mapsto t_{\psi_i}$ does not induce a group morphism $\UMiy_\chi(A, \I, \Omega) \to \UMiy_\chi(B, \J, \Omega')$.
\end{example}
This behavior is caused by the fact that we can distort the map $\psi$.
If we now restrict to \emph{axial} decomposition algebras (see section~\ref{se:axial}) that are sufficiently nice with respect to the Miyamoto maps,
then this type of distortion cannot occur, and $\UMiy_\Y$ becomes a functor.
\begin{definition}
    Let $(\Gamma, *)$ be a group fusion law, let $\Y \leq \X_R(\Gamma)$ be a subgroup of the $R$-character group
    and let $\Phi = (X, *)$ be a fusion law with a $\Gamma$\dash grading.
    Let $\lambda \colon X \to R$ be an evaluation map and let $(A, \I, \Omega, \alpha) \in \AXDEC$ be an axial decomposition algebra,
    with axes $a_i := \alpha(i)$ for each $i \in \I$. 
    By Proposition~\ref{pr:view}, we can also view this as a $\Gamma$\dash decomposition algebra (but usually \emph{not} as an axial $\Gamma$-decomposition algebra!).
    For each $i \in \I$ and each $\chi \in \Y$, let $\tau_{i,\chi}$ be the corresponding Miyamoto map.
    \begin{enumerate}[(i)]
        \item
            We call $(A, \I, \Omega, \alpha)$ \emph{Miyamoto-stable} with respect to $\Y$
            if for each $i \in \I$ and each $\chi \in \Y$, there is a permutation $\pi_{i,\chi}$ of $\I$ such that
            the pair $(\tau_{i,\chi}, \pi_{i,\chi})$ is an automorphism of $(A, \I, \Omega, \alpha)$ in $\AXDEC$.
            In other words, for each $i,j \in \I$:
            \begin{itemize}
                \item
                    the Miyamoto map $\tau_{i,\chi}$ permutes the axes; explicitly, $\tau_{i,\chi}(a_j) = a_{\pi_{i,\chi}(j)}$;
                \item
                    $\tau_{i,\chi}(A_x^j) = A_x^{\pi_{i,\chi}(j)}$ for each $x \in X$.
            \end{itemize}
            In particular, if $(A, \I, \Omega, \alpha)$ is Miyamoto-stable, then the $\Gamma$\dash decomposition algebra $(A, \I, \Omega)$ is Miyamoto-closed
            (see Definition~\ref{def:Miy}\eqref{Miy:closed}).
        \item
            We call $(A, \I, \Omega, \alpha)$ \emph{of unique type} with respect to $\Y$
            if both the map $\alpha \colon \I \to A$ and the map $\I \to \Hom(\Y, \Aut(A)) \colon i \mapsto (\chi \mapsto \tau_{i,\chi})$ are injective.
            In other words, for each $i \neq j$:
            \begin{itemize}
                \item $a_i \neq a_j$;
                \item there is at least one $\chi \in \Y$ such that $\tau_{i,\chi} \neq \tau_{j,\chi}$.
            \end{itemize}
            In particular, the assumption that $\alpha$ is injective implies that the permutations $\pi_{i,\chi}$ are now uniquely determined by $\tau_{i,\chi}$.
    \end{enumerate}
\end{definition}
\begin{theorem}\label{thm:axial-functor}
    Let $(\Gamma, *)$ be a group fusion law and let $\chi \colon \Gamma \to R^\times$ be a group homomorphism.
    Let $\Phi = (X, *)$ be a fusion law with a $\Gamma$-grading and let $\lambda \colon X \to R$ be an evaluation map.

    Let $\CAT$ be the full subcategory of $\AXDEC$ consisting of axial decomposition algebras that are Miyamoto-stable and of unique type with respect to $\Y$.
    Then $\UMiy_\Y \colon \CAT \to \GRP$ is a functor.
\end{theorem}
\begin{proof}
    Let $(A, \I, \Omega, \alpha) \xlongrightarrow{(\varphi, \psi)} (B, \J, \Omega', \beta)$ in $\CAT$.
    Notice that $(\varphi, \psi)$ is also a morphism in $\GDEC$.
    By Lemma~\ref{le:tau_i}, $\varphi \circ \tau_{i,\chi} = \tau_{\psi(i),\chi} \circ \varphi$ for all $i \in \I$ and all $\chi \in \Y$.
    For each $i \in \I$ and each $j \in \J$, we write $a_i := \alpha(i)$ and $b_j := \beta(j)$.
    Then for all $i,j \in \I$ and all $\chi \in \Y$, we have
    \begin{multline*}
        b_{\psi(\pi_{i,\chi}(j))} = \varphi(a_{\pi_{i,\chi}(j)}) = \varphi(\tau_{i,\chi}(a_j)) = \tau_{\psi(i),\chi}(\varphi(a_j)) \\
            = \tau_{\psi(i),\chi}(b_{\psi(j)}) = b_{\pi_{\psi(i),\chi}(\psi(j))} ,
    \end{multline*}
    and because $\beta$ is assumed to be injective, we get
    \begin{equation}\label{eq:psipi}
        \psi(\pi_{i,\chi}(j)) = \pi_{\psi(i),\chi}(\psi(j)) .
    \end{equation}
    We will show that the map $t_{i,\chi} \mapsto t_{\psi(i),\chi}$ induces a group morphism $\UMiy_\Y(A, \I, \Omega) \to \UMiy_\Y(B, \J, \Omega')$
    by showing that if $(j, k) \in R_{\tau_{i,\chi}}$, then also $(\psi(j), \psi(k)) \in R_{\tau_{\psi(i), \chi}}$.
    So let $(j, k) \in R_{\tau_{i,\chi}}$; then by Lemma~\ref{le:umiy}\eqref{umiy:id},
    \[ \tau_{k,\chi'} = \conj{\tau_{i,\chi}}{\tau_{j,\chi'}} = \tau_{\pi_{i,\chi}(j),\chi'} \]
    for all $\chi' \in \Y$.
    Because $(A, \I, \Omega, \alpha)$ is of unique type with respect to $\Y$, this can only happen if $k = \pi_{i,\chi}(j)$.
    Hence, by~\eqref{eq:psipi} and by Lemma~\ref{le:umiy}\eqref{umiy:id} again, also
    \[ \tau_{\psi(k),\chi'} = \tau_{\psi(\pi_{i,\chi}(j)),\chi'} = \tau_{\pi_{\psi(i),\chi}(\psi(j)),\chi'} = \conj{\tau_{\psi(i),\chi}}{\tau_{\psi(j),\chi'}} \]
    for all $\chi' \in \Y$.
    We conclude that indeed $(\psi(j), \psi(k)) \in R_{\tau_{\psi(i), \chi}}$.
\end{proof}

\begin{example}\label{ex:Miy-not-functorial}
    The previous theorem is false for the ordinary Miyamoto group $\Miy_\mathcal{Y}$, as we now illustrate.
    Let $n \geq 3$ be odd and consider the matrix algebra $M_n(k)$ of all $n\times n$-matrices over a field~$k$ with $\Char(k) \neq 2$. Let $J_n := M_n(k)^+$ be the corresponding Jordan algebra; this is the commutative non-associative algebra with multiplication $A \bullet B := \tfrac{1}{2}(AB + BA)$.
    
    Let $E_n$ be the set of all primitive idempotents of $J_n$. These are the matrices that are diagonalizable with eigenvalues $1$ with multiplicity $1$ and $0$ with multiplicity~$n-1$. It is well known that each idempotent $e$ in a Jordan algebra $J$ gives rise to a decomposition of $J$ into \textit{Peirce subspaces}, the eigenspaces of $\ad_e$ with eigenvalues $0$, $\tfrac{1}{2}$ and $1$, and moreover, this decomposition satisfies the Jordan fusion law from Example~\ref{ex:jordan} (for example see \cite{Jacobson}*{p.\@~119, Lemma 1}). 
    In the case of $J_n$ and $e \in E_n$, these eigenspaces have dimension $(n-1)^2$, $2(n-1)$ and $1$, respectively.
    This gives $J_n$ the structure of a primitive axial decomposition algebra $(J_n, E_n, \Omega, \id)$ admitting a $\Z/2\Z$-grading; it is clearly of unique type.
    
    For each $e \in E_n$, the corresponding Miyamoto map $\tau_e$ is precisely the conjugation action of $2e-1$ on $J_n$; since $n$ is odd, $2e-1 \in \SL_n(k)$. Hence the Miyamoto group $G = \Miy(J_n, E_n, \Omega)$ is isomorphic to the group generated by the elements $[2e-1] \in \PSL_n(k) \leq \Aut(J_n)$ for $e \in E_n$. Since $G$ is a non-trivial normal subgroup of $\PSL_n(k)$, it is isomorphic to $\PSL_n(k)$ itself.
    
    Now consider the algebra morphism
    \[ \varphi \colon J_n \to J_{n+2} \colon A \mapsto \begin{psmallmatrix} A & 0 \\ 0 & 0 \end{psmallmatrix} \]
    and the map $\psi \colon E_n \to E_{n+2}$ given by restriction of $\varphi$ to $E_n$.
    Then the pair $(\varphi, \psi)$ is a morphism of axial decomposition algebras.
    However, the map $\tau_e \mapsto \tau_{\psi(e)}$ does not extend to a group homomorphism from $\PSL_n(k)$ to $\PSL_{n+2}(k)$:
    the product of the Miyamoto maps corresponding to the primitive idempotents $E_{11},\dots,E_{nn}$ (where $E_{ij}$ is the matrix that is zero everywhere except at position $(i,j)$ where it has entry $1$) is trivial in $\PSL_n(k)$, but the product of their images under $\psi$ is equal to the element $\lbrack \operatorname{diag}(1, 1, \dots, 1, -1, -1) \rbrack \in \PSL_{n+2}(k)$.
   
    Notice that, in contrast, the \textit{universal} Miyamoto group always has a quotient isomorphic to $\SL_n(k)$.
    (Determining the precise structure of the universal Miyamoto group seems to be a challenging problem.)
\end{example}

\section{Decomposition algebras from representations}\label{se:rep}

In this section we will see how representation theory directly gives rise to interesting decomposition algebras.
We will assume that our base ring is the field $\C$ of complex numbers.

So let $A$ be any finite-dimensional $\C$-algebra.
Let $H$ be any finite subgroup of the automorphism group of $A$ and let $\Irr(H)$ be its representation fusion law
as in Example~\ref{ex:reprfus}.
If $A$ is semisimple as a $\C H$-module, then its unique decomposition into $H$\dash isotypic components will be an $\Irr(H)$-decomposition of $A$.

\begin{definition}[{See \cite{Ser77}}]
    Let $H$ be a finite group and let $A$ be a semisimple $\C H$-module.
	Let $V_1 \oplus \dots \oplus V_n$ be a decomposition of $A$ into irreducible modules.
	Denote the irreducible character of $H$ corresponding to $V_i$ by $\chi_i$.
	For each $\chi \in \Irr(H)$, the submodule
	\[
		A_\chi \coloneqq \bigoplus_{\chi_i = \chi} V_i
	\]
	is called the \textit{isotypic component} of $A$ corresponding to $\chi$.
	The decomposition
	\[ A = \bigoplus_{\chi \in \Irr(H)} A_\chi \]
	is called the \textit{$H$-isotypic decomposition} of $A$; it is uniquely determined by $A$ and $H$.
	The module $A$ is called \textit{multiplicity-free} if each isotypic component is irreducible; that is
	if $\chi_i \neq \chi_j$ for all $i \neq j$.
\end{definition}

\begin{theorem} \label{thm:constructdecomp}
	Let $A$ be a $\C$-algebra.
	Let $H$ be any finite subgroup of the automorphism group of $A$ and let $(\Irr(H),\ast)$ be its representation fusion law.
	Let $\{ H_i \mid i \in \I\}$ be a set of (some or all) conjugates of $H$ in $\Aut(A)$ indexed by some set $\I$.
	Then:
    \begin{enumerate}[\rm (i)]
        \item
        	The $H$-isotypic decomposition $A = \bigoplus_{\chi \in \Irr(H)} A_\chi$ of $A$ is an $(Irr(H),\ast)$\dash decomposition.
        \item        
        	If $A$ is multiplicity-free (as a $\C H$-module),
        	then any non-zero element $a \in A_1$ is an axis for this decomposition.
        \item
            For each $i \in \I$, let $A = \bigoplus_{\chi \in \Irr(H)} A^i_\chi$ be the $H_i$-isotypic decomposition of~$A$.
            Let $\Omega = \bigl( ( A_\chi^i )_{\chi \in \Irr(H)} \mid i \in \I \bigr)$.
            Then $(A, \I, \Omega)$ is an $(\Irr(H),\ast)$\dash decomposition algebra.
        \item
        	If $A$ is multiplicity-free (as a $\C H$-module) and
        	for each $i \in \I$, $a_i$ is a non-zero element of $A_1^i$.
        	Then $(A, \I, \Omega, \alpha)$ is an axial decomposition algebra, 
          where $\alpha \colon \I \to A \colon i \mapsto a_i$.
    \end{enumerate}
\end{theorem}
\begin{proof}
    \begin{enumerate}[\rm (i)]
        \item
        	Let $V_1 \oplus \dots \oplus V_n$ be a decomposition of $A$ into irreducibles.
          By Schur's lemma \cite{Ser77}*{\S 2.2}, $\Hom(V_i \otimes V_j , V_k) = 0$ whenever $\chi_k$ is not a constituent of $\chi_i \otimes \chi_j$. Hence the projection of $V_i \cdot V_j$ onto $V_k$ is zero.
        \item
        	Note that the requirement that $A$ is multiplicity-free implies that each $A_\chi$ is a simple $\C H$-module.
        	The fusion law implies that $A_1 A_\chi \subseteq A_\chi$ for all $\chi \in \Irr(H)$.
          Thus for any non-zero $a \in A_1$, we have 
          $(\ad_a)_{\textstyle{}_{A_\chi}} \in 
          \Hom(A_\chi,A_\chi)$.  Schur's lemma now implies that 
          $\Hom(A_\chi,A_\chi) \cong \C$ and hence $a$ is an axis for this 
          decomposition.
        \item
            This follows from (i).
        \item
            This follows from (ii).
        \qedhere
    \end{enumerate}
\end{proof}

\begin{example}
	A typical choice for $H$ is the centralizer $C_G(g)$ of an automorphism $g \in G := \Aut(A)$ of finite order $n$.
    Notice that by Proposition~\ref{pr:rep}, this implies that the fusion law $(\Irr(H), \ast)$ is $\Z/n\Z$-graded.
    Moreover, by restricting the characters of $H$ to $\langle g \rangle \leq Z(H)$, we see that the Miyamoto maps corresponding to
    the decomposition $A = \bigoplus_{\chi \in \Irr(H)} A_\chi$ with respect to this $\Z/n\Z$-grading
    are precisely the elements of $\langle g \rangle \leq \Aut(A)$.
    
	For example, if $A$ is the Griess algebra, we can recover its structure as an axial algebra (with $\Z/2\Z$-grading)
	by taking $H$ equal to the centralizer of a $2A$-involution and the Miyamoto group of this axial algebra is precisely
	the group generated by all those $2A$-involutions, i.e., the Monster group.
\end{example}

Conversely, we can use this technique to refine the fusion law of a decomposition.

\begin{proposition} \label{prop:fusionlawfromdecomp}
	Let $A = \bigoplus_{x \in X} A_x$ be a decomposition of a $\C$-algebra $A$.
	Let $H \leq \Aut(A)$ be a finite subgroup such that each $A_x$ is $H$-invariant.
	For each $x \in X$, let $\chi_x$ be the character of the $\C H$-module $A_x$. Consider the map
	\[
		 \ast \colon X \times X \to 2^X \colon (x,y) \mapsto \{z \in X \mid \langle \chi_z , \chi_x\chi_y \rangle \neq 0 \}
	\]
	where $\langle \; , \; \rangle$ is the inner product on the space of class functions of $H$.
	Then $A = \bigoplus_{x \in X} A_x$ is an $(X,\ast)$-decomposition of $A$.
\end{proposition}
\begin{proof}
	This follows immediately from the fact that $\Hom_{\C H}(A_x \otimes A_y , A_z) = 0$ whenever $\langle \chi_z, \chi_x\chi_y \rangle = 0$.
\end{proof}

\begin{remark}
    Although we formulated the results in this section for a finite group~$H$, they can easily be generalized to Lie groups or linear algebraic groups.
	The proof only requires a suitable version of semi-simplicity and Schur's lemma.
\end{remark}

\section{Norton algebras} \label{se:Nor}

If $G$ is the Miyamoto group of a Miyamoto-stable axial decomposition algebra, then $G$ has a natural permutation action on the set of axes. We give a reverse construction. Starting from a transitive permutation representation of a group~$G$, we construct a Miyamoto-stable axial decomposition algebra on which $G$ acts by automorphisms. More precisely, we will prove that Norton algebras are axial decomposition algebras. Norton algebras, in the sense of this section, were first introduced in \cite{CGS} starting from association schemes. We refer to \cite{BI84} for more information about association schemes and Norton algebras.

\begin{definition}
	Let $X$ be a finite set and let $R_i \subseteq X \times X$ for $i = 0,\dots,d$. Assume:
		\begin{enumerate}[(I)]
      \item $X \times X = R_0 \cup \dots \cup R_d$ and $R_i \cap R_j = 
        \emptyset$ for all $i \neq j$; that is, the sets $R_i$ form a partition of 
        $X \times X$;
			\item $R_0 = \{(x,x) \mid x \in X\}$;
      \item for each $i$, $\prescript{t}{}{\!R_i} \coloneqq \{(x,y) \mid (y,x) \in R_i\} = R_{i'}$ for some $i'$;
			\item for any $(x,y) \in R_k$, the number of $z \in X$ for which $(x,z) \in R_i$ and $(z,y) \in R_j$ is a constant $p_{ij}^k$ only depending on $i,j,k$;
			\item $p_{ij}^k = p_{ji}^k$ for all $i,j,k$.
		\end{enumerate}
	Then $(X,\{R_i\}_{0\leq i \leq d})$ is called a (commutative) \emph{association scheme}.
  If $\prescript{t}{}{\!R_i} = R_i$ for all $i$, then we call the association scheme \emph{symmetric}.
\end{definition}

\begin{example}[\cite{BI84}*{\S II.2, Example 2.1}]\label{ex:assocfromgroup}
  Let $G$ be a transitive permutation group acting on a finite set $\Omega$. 
Denote the orbits of $G$ on $\Omega \times \Omega$ by 
$\Lambda_0,\dots,\Lambda_d$ where $\Lambda_0 = \{(x,x) \mid x \in \Omega\}$. 
Then $(\Omega,\{\Lambda_i\}_{0 \leq i \leq d})$ satisfies (I)--(IV). Requirement 
(V) is satisfied if and only if the corresponding permutation character is 
multiplicity free. This association scheme is 
symmetric if and only if for any $i$ and for any $x,y \in \Lambda_i$  there 
exists a $g \in G$ such that $\conj{g}{x} = y$ and $\conj{g}{y} = x$. If this condition is 
satisfied, we say that $G$ 
acts generously transitively on $\Omega$.
\end{example}

\begin{definition} \label{def:Nortonalgebras}
	Let $\mathcal{X} = (X,\{R_i\}_{0 \leq i \leq d})$ be an association scheme.
	\begin{enumerate}[(i)]
		\item For each $i$, let $A_i$ be the matrix whose rows and columns are indexed by the set $X$ and such that
		\[
			(A_i)_{xy} = \begin{cases}
				0 &\text{if $(x,y) \notin R_i$,} \\
				1 &\text{if $(x,y) \in R_i$.}
			\end{cases}
		\]
		Then $A_0 = I$ and $A_iA_j = \sum_{k=0}^d p_{ij}^kA_k$ for all $i,j$. Hence, by (V), they span a commutative subalgebra of the full matrix algebra. This algebra is called the \textit{Bose--Mesner algebra} or the \textit{adjacency algebra}. This algebra is also closed under the entry-wise or \textit{Hadamard matrix product} which we denote by $\circ$: $(A \circ B)_{ij} = (A_{ij}B_{ij})$.
		\item Let $V$ be the Hermitian space with orthonormal basis $\{e_x \mid x \in X \}$ indexed by the set $X$. Then the $A_i$ act naturally on $V$ and because they pairwise commute, they can be diagonalized simultaneously by a unitary matrix $U$. Let $V = V_0 \oplus V_1 \oplus \dots \oplus V_r$ be the decomposition of $V$ into common eigenspaces. It is readily verified that we can pick $V_0 = \langle (1,\dots,1) \rangle$. Denote the matrix form, with respect to the basis $\{e_x \mid x \in X\}$, of the projection $\pi_i$ of $V$ onto $V_i$ by $E_i$. Then $r=d$ and $E_0,\dots,E_d$ form a basis of primitive idempotents for the adjacency algebra of $\mathcal{X}$ \cite{BI84}*{\S 2.3, Theorem 3.1}. Since the adjacency algebra is closed under the Hadamard product, there exist constants $q_{ij}^k$ such that $E_i \circ E_j = \frac{1}{\left| X \right|}\sum_{k=0}^d q_{ij}^k E_k$. We call $q_{ij}^k$ the \emph{Krein parameters} of $\mathcal{X}$.
		\item For each $i,j$ and $k$ we can define a bilinear map $\sigma_{ij}^k: V_i \times V_j \to V_k$ as point-wise multiplication with respect to the basis $\{e_x \mid x \in X\}$ composed with projection onto $V_k$. That is,
		\[
			\sigma_{ij}^k(v,w) := \sum_{x \in X} \langle v , e_x \rangle \langle w , e_x \rangle \pi_k(e_x).
		\]
		In particular, $\sigma_{ii}^i$ gives $V_i$ the structure of a commutative non-associative algebra, which is called a \textit{Norton algebra}. We denote this product on $V_i$ by~$\star$.
	\end{enumerate}
\end{definition}

\begin{remark}
	If $\mathcal{X}$ is a symmetric association scheme then all the matrices $A_i$ will be symmetric and hence simultaneously diagonalizable by a real orthogonal matrix. In that case the matrices $E_i$ will be symmetric real matrices and the Norton algebras can be defined over $\mathbb{R}$.
\end{remark}

\begin{proposition}[\cite{BI84}*{\S II.8, Proposition 8.3}]
	We have
	\begin{enumerate}[\rm (i)]
		\item $\sigma_{ij}^k = 0$ if and only if $q_{ij}^k = 0$;
		\item $\sigma_{ij}^k(\pi_i(e_x),\pi_j(e_x)) = \frac{1}{\left| X \right|} q_{ij}^k \pi_k(e_x)$.
	\end{enumerate}
\end{proposition}
\begin{proof}
	This is readily verified from $E_k(E_i \circ E_j) = \frac{1}{\left| X \right|}q_{ij}^k E_k$.
\end{proof}

Norton algebras provide a rich source of examples of decomposition algebras:

\begin{theorem}
  Let $\mathcal{X} = (X,\{R_i\}_{0 \leq i \leq d})$ be a symmetric association 
  scheme. Let $e_x$ and $\pi_i$ be as in Definition \ref{def:Nortonalgebras}. 
  Let $V_i$ be one of its Norton algebras and suppose $\pi_i(e_x)$ is non-zero 
  for all $x \in X$. Then for each $x \in X$,
  \[
    \ad_{\pi_i(e_x)} \colon V_i \to V_i \colon v \mapsto \pi_i(e_x) \star v
  \]
  is diagonalizable. Let $\bigoplus_{\lambda \in \Lambda} (V_i)^x_\lambda$ be 
  the decomposition of $V_i$ into eigenspaces for $\ad_{\pi_i(e_x)}$.
  Let $\Omega \coloneqq \bigl( ((V_i)_\lambda^x)_{\lambda \in \Lambda} \mid x 
  \in X \bigr)$.
  Then $(V_i,X,\Omega,x \mapsto \pi_i(e_x))$ is an axial decomposition algebra.
\end{theorem}
\begin{proof}
  Consider the linear operator
  \[
    \theta \colon V \to V \colon v \mapsto \sum_{y \in X} \langle \pi_i(e_x) , 
  e_y \rangle \langle \pi_i(v) , e_y \rangle \, \pi_i(e_y).
  \]
  Its restriction to $V_i$ equals $\iota \circ \ad_{\pi_i(e_x)}$, where $\iota 
  \colon V_i \to V$ is the natural embedding. Since $V_i$ is an invariant 
  subspace of $\theta$, it suffices to prove that $\theta$ is diagonalizable. 
  The matrix form of $\theta$ with respect to the basis $\{e_x \mid x \in X \}$ 
  is $E_i \operatorname{diag}(\pi_i(e_x)) E_i$.
  Since $\mathcal{X}$ is symmetric, this is a real symmetric matrix and hence 
  $\theta$ is a Hermitian operator on~$V$ and therefore $\theta$ is 
  diagonalizable. The remaining statement is obvious.
\end{proof}

\begin{remark}
	If $\mathcal{X}$ is not symmetric, then $\theta$ will not necessarily be a Hermitian operator. However, it can still be interesting to look at the decomposition of $V_i$ into generalized eigenspaces of $\ad_{\pi_i(e_x)}$.
\end{remark}

In the next example, we illustrate how to obtain a suitable fusion law using Proposition \ref{prop:fusionlawfromdecomp} and Theorem \ref{thm:constructdecomp}.

\begin{example}
  Let $G$ be a group and $X$ a conjugacy class of elements of order~$n$. Suppose 
  that $G$ acts generously transitively on $X$ and consider the corresponding 
  symmetric association scheme. Let $V_i$ be one of its Norton algebras. The 
  natural permutation action of $G$ on $X$ induces algebra automorphisms on this 
  Norton algebra. Hence there exists a morphism $\rho\colon G \to \Aut(V_i) \leq 
  \GL(V_i)$.
	Let $C_G(x)$ be the centralizer in $G$ of $x \in X$ and $(\Irr(C_G(x)),\ast)$ its representation fusion law. Since the action of $C_G(x)$ commutes with the linear operator $\ad_{\pi_i(e_x)}$, it leaves invariant its eigenspaces. Now apply Proposition \ref{prop:fusionlawfromdecomp} with $H = C_G(x)$ to construct a fusion law $(\Lambda,\ast')$ for the decomposition $\bigoplus_{\lambda \in \Lambda} (V_i)^x_\lambda$ of $V_i$. Now let $(V_i)^x_\lambda = \bigoplus_{j \in J} (V_i)^x_{\lambda,j}$ be the decomposition of $(V_i)^x_\lambda$ into irreducible subrepresentation for $C_G(x)$. Denote the irreducible character of $C_G(x)$ corresponding to $(V_i)^x_{\lambda,j}$ as $\chi_j$ and define
	\[
		(V_i)^x_{\lambda,\chi} = \bigoplus_{\chi_j = \chi} (V_i)^x_{\lambda,j}.
	\]
	By Theorem \ref{thm:constructdecomp}, the decomposition
	\[ V_i = \bigoplus_{\substack{\lambda \in \Lambda \\ \chi \in \Irr(C_G(x))}} (V_i)^x_{\lambda,\chi} \]
	is a $(\Lambda \times \Irr(C_G(x)) , \bullet)$-decomposition, where $(\Lambda \times \Irr(C_G(x)) , \bullet)$ is the product of the fusion laws
	$(\Lambda,\ast')$ and $(\Irr(C_G(x)),\ast)$.
	Since $\langle x \rangle \leq Z(C_G(x))$ the map $(\lambda,\chi) \mapsto \chi(x)$ defines a $\Z/n\Z$ grading of this fusion law.
	The Miyamoto involution $\tau_x$ with respect to this $\Z/n\Z$-grading is precisely the automorphism $\rho(x)$ and the Miyamoto group is $\langle \rho(x) \mid x \in X \rangle = \rho\bigl(\langle X \rangle \bigr)$. In particular, if $G$ is simple, then the Miyamoto group coincides with $\rho(G) \cong G$.
\end{example}

\appendix
\section{The category of decomposition algebras}\label{se:appendix}

We now explore some more advanced categorical properties of decomposition algebras.
The reader who is less acquainted with the terminology and concepts may consult the excellent text book by Tom Leinster \cite{Leinster}.

\smallskip

Fix a commutative ring $R$ and a fusion law $\Phi = (X, *)$ and let $\DEC$ be as in Definition~\ref{def:DEC}.

\begin{remark}\label{rmk:adjoint}
    The category $\DEC$ has an initial object $(0, \emptyset, \emptyset)$ and a terminal object $(0, \{ * \}, (0))$.
    This category admits two obvious forgetful functors, namely
    \begin{align*}
        \DEC &\to \ALG \colon (A, \I, \Omega) \rightsquigarrow A \qquad \text{and} \\
        \DEC &\to \SET \colon (A, \I, \Omega) \rightsquigarrow \I .
    \end{align*}
    The corresponding left adjoints are given by
    \begin{align*}
        \ALG &\to \DEC \colon A \rightsquigarrow (A, \emptyset, \emptyset) \qquad \text{and} \\
        \SET &\to \DEC \colon \I \rightsquigarrow \bigl( 0, \I, (0 \mid i \in \I) \bigr) ,
    \end{align*}
    respectively.
\end{remark}

\begin{proposition}
  The category $\DEC$ is complete.
\end{proposition}
\begin{proof}
  Recall that a category is complete if it contains all (small) limits. From 
  the existence theorem for limits it is sufficient to show that $\DEC$
  has equalizers and all products; see, e.g., \cite[Proposition 5.1.26]{Leinster}.

  \def\Aj{(A_j, \I_j, \Omega_j)}
  We begin by showing the existence of products. Let $\Aj$ be a set of 
  decomposition algebras indexed by some set $J$.
  The forgetful functors of Remark~\ref{rmk:adjoint} preserve limits and hence 
  if the product of $\Aj$ exists it must consist of the algebra $\prod_{j\in J} 
  A_j$ and the index set $\prod_{j\in J} \I_j$. Let $\Pi$ be a set of 
  decompositions indexed by $\prod_{j\in J}\I_j$, where
  \[
    \Pi[(i_j)_{j\in J}] = \Bigl(\prod_{j\in J} (A_j)^{\textstyle 
    {}^{i_j}}_{\textstyle {}_x}\,\Big\vert\, x\in X\Bigr)
  \]
  Let $\pi_k\colon \prod_{j\in J} A_j \to A_k$ and $\psi_k\colon \prod_{j\in J} 
  \I_j \to \I_k$ be the natural projections of algebras and sets respectively: 
  we will show that $(\pi_k,\psi_k)$ is the product of the set of decomposition 
  algebras \(\Aj\). 

  Firstly, if \(\mathbf i=(i_j)_{j\in J} \in \prod_{j\in J}\I_j\) and \(x\in X\) 
  then
  \[
    \pi_k\bigl(\Pi[\mathbf i]_x\bigr) = \pi_k\Bigl(\prod_{j\in 
    J}(A_j)_x^{\textstyle{}^{i_j}}\Bigr) = (A_k)_x^{\textstyle{}^{i_k}} 
    = (A_k)_x^{\textstyle{}^{\psi_k(\mathbf i)}} 
  \]
  and so $(\pi_k,\psi_k)$ is a morphism in $\DEC$.

  Next we need to show that for any cone \((\varphi_j,\theta_j)\colon (B,\K, 
  \Sigma) \rightarrow (A_j,\I_j,\Omega_j)\) there is a unique morphism from 
  \((B,\K,\Sigma)\) to the product making the following diagram commute
  \[\begin{tikzcd}
    & (B,\K,\Sigma) \ar[ddl, "{(\varphi_1,\theta_1)}" {above left}] 
                    \ar[ddr, "{(\varphi_J,\theta_J)}" {above right}] 
                    \ar[d,dashed] \\ 
    & {\rm Product}    \ar[dl, "{(\pi_1,\psi_1)}" {below right,xshift=-1ex}] 
                    \ar[dr, "{(\pi_J,\psi_J)}" {below left, xshift= 1ex}] \\
    (A_1,\I_1,\Omega_1) & \cdots & (A_J,\I_J,\Omega_J)
  \end{tikzcd}.\]
  
  If \(b\in B^k_x\) then \(\varphi_j(b)\in 
  (A_j)^{{\scriptstyle \theta_j(k)}}_{{}_{\scriptstyle x}}\) for all \(j\in J\) 
  and hence \( \left(\varphi_j(b)\right)_{j\in J} \in \prod_{j\in J} 
  (A_j)^{{\scriptstyle \theta_j(k)}}_{{}_{\scriptstyle x}}\). This shows that 
  the obvious map from \((B,\K,\Sigma)\) to the product is actually a morphism 
  in \(\DEC\). This map clearly makes the diagram commute and the uniqueness is 
  a consequence of the uniqueness of \(\pi_j\) and \(\psi_j\) in their 
  respective categories. This completes the proof of the existence of products.

  We now show that equalizers exist in \(\DEC\).
  Let $(\varphi_1,\psi_1)$ and $(\varphi_2,\psi_2)$ be two morphisms of $\DEC$: 
  \[(\varphi_1,\psi_1), (\varphi_2,\psi_2)\colon (A,\I,\Omega)\to(B,\J,\Theta).
  \]  Let $\varphi\colon E\to A$ be the equalizer of $\varphi_1$ and 
  $\varphi_2$ in $\ALG$, let $\psi\colon \K\to\I$ be the equalizer of 
  $\psi_1$ and $\psi_2$ in $\SET$ and let $\Sigma$ be the tuple of 
  decompositions given by
  \[
    \Sigma[k] = \left( \varphi^{-1}\Bigl( A^{\psi(k)}_x \Bigr) \,\middle\vert\, 
    x\in X \right) \text{ for \(k\in \K\).}
  \]
  To see that this is indeed a tuple of decompositions: firstly, if \(e \in 
  E^k_x \cap \sum_{y\neq x}E^k_y\) then \(\varphi(e) \in A^{\psi(k)}_x \cap 
  \sum_{y\neq x}A^{\psi(k)}_y = 0\).  Now since equalizers are monic we must 
  have \(e=0\). Secondly, if \(e\in E\) and \(k\in \K\) then \(\varphi(e) = 
  \sum_{x\in X} a_x\) for some \(a_x \in A^{\psi(k)}_x\). It is sufficient to 
  show that each \(a_x\) is in the image of \(\varphi\). As \(e\in E\) we know 
  that \(\varphi_1(e)=\varphi_2(e)\) and hence \(\sum_{x\in X} 
  \left(\varphi_1(a_x) - \varphi_2(a_x)\right)=0\). However \(k\in \K\)  implies 
  that each term \(\varphi_1(a_x)-\varphi_2(a_x)\) is in a distinct component of 
  a direct sum and hence each is zero. Now since \(\varphi_1\) and \(\varphi_2\) 
  act equally on \(a_x\) for each \(x\in X\), each \(a_x\) must have a preimage 
  in \(E\).

  It is clear from the definition that $(\varphi,\psi)$ is a morphism of $\DEC$ 
  so we need only check that it is the equalizer of \((\varphi_1,\psi_1)\) and 
  \((\varphi_2,\psi_2)\). Let \((\gamma,\tau)\colon (F,\L,\Phi) \to 
  (A,\I,\Omega)\) be a morphism such that \((\varphi_1,\psi_1)\circ(\gamma,\tau) 
  = (\varphi_2,\psi_2)\circ(\gamma,\tau)\). Define \((\delta,\sigma)\) by
  \begin{align*}
    \delta\colon F &\to E & \sigma\colon \L&\to \K \\
    f &\mapsto \varphi^{-1}(\gamma(f)) &
    l &\mapsto \psi^{-1}(\tau(l)).
  \end{align*}
  Then \((\varphi,\psi)\) is a morphism of decomposition algebras and 
  \((\varphi,\psi)\circ(\delta,\sigma)=(\gamma,\tau)\). Uniqueness again follows 
  from the uniqueness of \(\varphi\) and \(\psi\) in \(\ALG\) and \(\SET\) 
  respectively. This completes the proof that  equalizers exist in \(\DEC\) and 
  hence that \(\DEC\) is complete.
\end{proof}

We now turn our attention to ideals and quotients of decomposition algebras.

\begin{definition}\label{def:ideal}
    \begin{enumerate}[(i)]
        \item
            Let $(A, \I, \Omega)$ be a decomposition algebra and let $I \unlhd A$ be an algebra ideal.
            For each $i \in \I$ and each $x \in X$, let $I_x^i := A_x^i \cap I$
            and let $\Omega \cap I := \bigl( (I_x^i)_{x \in X} \mid i \in \I \bigr)$.
            We call $I$ a \emph{decomposition ideal} of $(A, \I, \Omega)$ if for each $i \in \I$, we have $I = \bigoplus_{x \in X} I_x^i$.
            Notice that this implies that $(I,\I,\Omega \cap I)$ is an object in $\DEC$.
        \item\label{def:ideal:quo}
            If $I$ is a decomposition ideal of $(A, \I, \Omega)$ and $B = A/I$, then $(B, \I, \Sigma)$ is again a decomposition algebra
            (which we then call the \emph{quotient} decomposition algebra) obtained by setting
            \[ B_x^i := (A_x^i + I)/I \]
            for all $i \in \I$ and all $x \in X$, and then letting $\Sigma = \bigl( ( B_x^i )_{x \in X} \mid i \in \I \bigr)$.
            Notice that the condition $I = \bigoplus_{x \in X} I_x^i$ ensures that the sum $\sum_{x \in X} B_x^i$ is a direct sum.
    \end{enumerate}
\end{definition}

\begin{proposition}\label{prop:kernel}
  Let $(\varphi,\psi)\colon (A, \I, \Omega_A) \to (B, \J, \Omega_B)$
  be a morphism of decomposition algebras.  Then $K = \ker\varphi$
  is a decomposition ideal of $(A, \I, \Omega_A)$ with corresponding quotient $(A/K,\I,\Sigma)$ as in Definition~\ref{def:ideal}(\ref{def:ideal:quo}) above.

  Conversely, if $I$ is a decomposition ideal of $(A, \I, \Omega)$ and
  $\pi\colon A \onto A/I$ is the natural projection of algebras, then
  $(I, \I, \Omega \cap I)$ is the equalizer of the epimorphism
  $(\pi,\id)\colon (A, \I, \Omega) \onto (A/I, \I, \Sigma)$ and the morphism
  $(0,\id)$.
\end{proposition}
\begin{proof}
  We begin by showing that \(K=\ker\varphi\) is a decomposition ideal.
  Fix some \(i \in \I\) and let \(K_x^i = K \cap A_x^i\). It is clear that \( K_x^i \cap \sum_{y \neq x} K_y^i=0 \)
  for all \(x\in X\) and that \(K\supseteq \sum_{x\in X} K_x^i\),
  thus we need only show the opposite inclusion. For any
  \(k\in K\) we may write \(k = \sum_{x\in X} a_x^i\), where each
  \(a_x^i\in A_x^i\). It is sufficient to show that \(a_x^i\in K\),
  but
  \[
    \sum_{x\in X} \varphi(a_x^i) = \varphi(k) = 0
  \]
  where each \(\varphi(a_x^i) \in B_x^{\psi(i)}\) is in a different
  component of a direct sum. Hence \(\varphi(a_x^i) = 0\) for all \(x\).

  The second part follows directly from the first part once we note
  that \(I\) is the algebra kernel of \(\pi\).
\end{proof}

\begin{remark}
  Recall that the categorical definition of a kernel of a morphism is the 
  equalizer of the given morphism and a zero morphism. We would like to be able 
  to refer to the decomposition ideal $(I,\I,\Omega\cap I)$ in 
  Proposition~\ref{prop:kernel} as the kernel of the projection, however since 
  the category $\DEC$ does not contain zero morphisms the definition of kernel 
  does not make sense. Instead, in Proposition~\ref{prop:kernel}, we use 
  \((0,\id)\) in place of the zero morphism and in this sense the decomposition 
  ideals (as equalizers of these morphisms) are as close to kernels as we can 
  realistically achieve.
\end{remark}

\clearpage

\bibliographystyle{alpha}

\vfill


\end{document}